\documentclass{jloganal}    

\title{Finitely additive measures on $\mathbb Z$ and additive combinatorics}

\author{Zeinab Ashtab}
\givenname{Zeinab}
\surname{Ashtab}
\address{Escuela Superior de F\'{\i}sica y Matem\'aticas\\ Instituto Polit\'ecnico Nacional\\ Av. Instituto Polit\'ecnico Nacional s/n Edificio 9 \\ Col. San Pedro Zacatenco\\ Alcald\'ia Gustavo A. Madero\\ 07738\\ CDMX\\ M\'exico.}
\email{zashtab@ipn.mx}
\author{David Fernández-Bretón}
\givenname{David J.}
\surname{Fernández-Bretón}
\email{dfernandezb@ipn.mx}
\urladdr{https://dfernandezb.web.app/english.html}

\arxivreference{}
\arxivpassword{}

\volumenumber{}
\issuenumber{}
\publicationyear{}
\papernumber{}
\startpage{}
\endpage{}
\doi{}
\MR{}
\Zbl{}
\received{}
\revised{}
\accepted{}
\published{}
\publishedonline{}
\proposed{}
\seconded{}
\corresponding{}
\editor{}
\version{}

\newtheorem{thm}{Theorem}[section]    % Standard theorem environment
\newtheorem{lem}[thm]{Lemma}          % Lemma environment with numbering 
\newtheorem{prop}[thm]{Proposition}    % Proposition environment with numbering 
\newtheorem{cor}[thm]{Corollary}          % Lemma environment with numbering 
\newtheorem{question}[thm]{Question} 
\theoremstyle{definition}
\newtheorem{defn}[thm]{Definition}    % Definition environment with 
\newtheorem{ex}[thm]{Example} 
\newtheorem*{rem}{Remark}             % Unnumbered environment for remarks.
\DeclareMathOperator{\fs}{FS}

\newcommand{\ba}{\mathop{\mathrm{ba}}}
\newcommand{\ip}{\mathnormal{\mathsf{IP}}}
\newcommand{\sip}{\mathnormal{\mathsf{sIP}}}

\begin{document}

\begin{abstract}    % type your abstract below
   We study (bounded) finitely additive measures on the group of integers $\mathbb Z$, as elements of the Banach algebra $\ba(\mathbb Z)$, viewed as a natural generalization of ultrafilters. The algebraic structure of $\ba(\mathbb Z)$ extends the semigroup structure of the \v{C}ech--Stone compactification, allowing methods from ultrafilter theory to be applied in a broader measure-theoretic setting. We investigate idempotent finitely additive measures and establish additive properties of subsets of $\mathbb Z$ having positive measure. We then proceed to study almost translation-invariant and translation-invariant finitely additive measures, showing that these stronger notions yield correspondingly stronger additive conclusions. In particular, we prove that every subset of $\mathbb Z$ whose measure exceeds a certain explicit threshold necessarily is an $\ip_{n}$-set; with stronger properties and lower thresholds depending on the properties of the relevant measures. Several examples illustrating the sharpness and limitations of the results are also presented, together with a discussion of open problems and directions for future research.
\end{abstract}

\maketitle

%%%%%%%%%%%%%%%%%%%%   Start of main body of article

\section{Introduction}

This paper deals with bounded finitely additive measures over the group of integers $\mathbb Z$, especially insofar as these measures constitute a generalization of the concept of an ultrafilter. The theory of bounded finitely additive measures is classical and is developed, for example, in~\cite{Dunford-Schwartz}. The set of all (signed) bounded finitely additive measures over $\mathbb Z$, denoted $\ba(\mathbb Z)$, has a structure of Banach algebra making it isomorphic to $\ell^\infty(\mathbb Z)^*$, and contains within its unit ball a copy of the \v{C}ech--Stone compactification $\beta\mathbb Z$. The algebraic structure of $\beta\mathbb Z$, including its interpretation as the space of ultrafilters endowed with the Stone topology and its natural semigroup operation, is treated in detail in~\cite{Hindman-Strauss} and in~\cite{comfort-negrepontis}. The interpretation of ultrafilters as a particular case of finitely additive measures is made precise by identifying ultrafilters with $\{0,1\}$-valued measures.\\
%**********************************
Under this optics, it turns out that the usual ultrafilter addition operation, extending the usual addition in $\mathbb Z$, and which is one of the fundamental tools in the algebraic approach to Ramsey theory, corresponds exactly to the Banach algebra (Arens) product in $\ba(\mathbb Z)$. Therefore the structure $\ba(\mathbb Z)$ can be conceived as a generalization of $\beta\mathbb Z$. In this paper we strive to correspondingly generalize the use of the additive tools of $\beta\mathbb Z$, useful in proving Hindman's theorem~\cite{hindman-thm} and many of its extensions. Hindman's theorem and the associated theory of finite-sums sets and IP-sets have become central tools in Ramsey theory and combinatorial number theory through the work of~\cite{hindman-ultrafiltersch},~\cite{Furstenberg2}, and \cite{Hindman-Strauss}.\\
We begin by studying idempotent finitely additive measures, and attempt to extract consequences of these kinds of measures in terms of finite sum-sets contained within subsets of $\mathbb Z$ with certain measures, by analogy with the ultrafilter case, where idempotent ultrafilters are central in extracting these kinds of conclusions.% On the other hand, we also study translation-invariant finitely additive measures (or invariant means), which have been extensively studied in the theory of amenable groups and harmonic analysis and, in particular, on the natural numbers in \cite{Gardner}.\\

Afterwards, we attempt to obtain variations of the concept of idempotent measures that could allow us to guarantee additive properties of sets based on their measure. We begin by studying {\it almost translation-invariant measures}, followed by {\it translation-invariant measures}. 
Translation-invariant finitely additive measures have been extensively studied in the theory of invariant means and amenable groups~\cite{F.Greenleaf}, particularly in the setting of the natural numbers~\cite{Gardner}.  Almost translation-invariant measures are, in the ultrafilter context, precisely equivalent to idempotence~\cite{Hindman-Strauss}; in this more generalized context they provide us with slightly better upper bounds than idempotent measures do. Translation-invariant measures, on the other hand, have stronger properties. We correspondingly introduce a stronger definition on the additive side, singling out sets of integers possessing sets of finite sums in a very strong way.

On Section \ref{Preliminaries} we lay out the basic definitions and results that we will take as a starting point. Then, on Section \ref{IdempotentandIPset}, we study idempotent finitely additive measures, proving additive results but also showing several examples. Section \ref{Translationinvariantmeasures} focuses on translation-invariant and almost translation-invariant measures, and in this section we prove the corresponding combinatorial theorems, analogous to the ones from the previous section, but with better bounds. Finally, there is a short final Section with some discussion and open questions.

\section{Preliminaries}
\label{Preliminaries}

In this section we state the definitions that will be used throughout the paper, as well as certain facts about them. All of these are standard, and can be found in classical sources such as~\cite{comfort-negrepontis,Hindman-Strauss,dales-automatic,Dunford-Schwartz}. We begin by stating what pertains to the theory of ultrafilters.

%***************************************************************
\begin{defn}
Let $X$ be an infinite set.
\begin{enumerate}
    \item An {\bf ultrafilter} on $X$ is a collection $u\subseteq P(X)$ satisfying the following axioms: 
    \begin{enumerate}
    \item 
    $\varnothing\in u$,
    \item
    If $A,B\in u$, then $A\cap B\in u$,
    \item 
    If $A\in u $ and $A\subseteq B\subseteq X$, then $B\in u$,
    \item 
    For every $A\subseteq X$
    \[
    A\in u \quad\text{or}\quad X\setminus A\in u,
    \]
    but not both.
    \end{enumerate}
    \item We denote the set of all ultrafilters on $X$ by $\beta X$.
    \item For each subset $A\subseteq X$ we define 
 \[
 \overline{A}=\left\{p\in \beta X\big| A\in p\right\}.
 \]   
 \item The {\bf Stone topology} on $\beta X$ is the topology generated by the family $\left\{\overline{A}\big| A\subseteq X\right\}$; this family forms a basis of clopen sets.
\end{enumerate}
\end{defn}

It is standard to consider $X$ to be a subset of $\beta X$ by means of identifying every $x\in X$ with the corresponding {\it principal ultrafilter} $\left\{A\subseteq X\big|x\in A\right\}$. In this line of thought, the topological space $\beta X$ (considered with the Stone topology) is a compact Hausdorff space containing $X$ (considered as a discrete space) as a dense open subset. Under these conditions, $\beta X$ is known as the {\it \v{C}ech--Stone compactification of $X$}. In the case where $X$ has a group structure, such as if $X=\mathbb Z$ equipped with addition, this operation can be extended to all of $\beta\mathbb Z$.

\begin{defn}
    The binary operation $+$ on $\beta\mathbb Z$ is given by: for $p,q\in \beta\mathbb Z$, the element $p+q\in\beta\mathbb Z$ is defined by 
 \[
 A\in p+q \quad\Leftrightarrow\quad \left\{n\in\mathbb Z\big|A-n\in q\right\}\in p.
 \]
 Equivalently,
 \[
 A\in p+q \quad\Leftrightarrow\quad \left\{n\in \mathbb Z\big|\left\{m\in \mathbb Z\big| n+m\in A\right\}\in q\right\}\in p.
 \]
\end{defn}

It is known that $+$ is a semigroup operation on $\beta\mathbb Z$, extending the addition on $\mathbb Z$. Note that $+$ is not commutative in $\beta\mathbb Z$. Under these conditions, $\beta\mathbb Z$ is known as a {\it right-topological semigroup}, a semigroup with the property that, for every fixed $v$, the mapping $u\longmapsto u+v$ is continuous.

Having defined all of the concepts related to ultrafilters, we proceed to state some definitions that can be seen as a generalization. 
\begin{defn}
    Let $X$ be a nonempty set.
    \begin{enumerate}
        \item A function $\mu\colon \wp(X)\longrightarrow [-\infty,\infty]$ is called a (signed) {\bf finitely additive measure} if it satisfies the following:
\begin{enumerate}
    \item 
    $\mu(\varnothing)=0$,
    \item 
    Whenever $A,B\subseteq X$ are disjoint, 
    \[
    \mu(A\cap B)=\mu(A) + \mu(B).
    \]
\end{enumerate}
    \item We will say that a finitely additive measure $\mu$ on $X$ is {\bf positive} if $\mu(A)\geq 0$ for all $A\subseteq X$.
    \item We use the notation $\ba(X)$ to denote the set of all bounded signed finitely additive measures on $X$.
    \end{enumerate}
\end{defn}

Note that a signed finitely additive measure $\mu$ admits a {\it Jordan decomposition} as $\mu=\mu^+-\mu^-$, where $\mu^+,\mu^-$ are both positive measures. However, $\mu$ need not necessarily admit a Riesz decomposition.

The set $\ba(X)$ can be thought of as the dual of the Banach space $\ell^\infty(X)$ by identifying each $\mu\in\ba(X)$ with the linear functional $:\ell^\infty(X)\longrightarrow\mathbb R$ given by integrating with respect to $\mu$. This perspective gives $\ba(X)$ the structure of a Banach space, where the norm for a positive finitely additive measure $\mu$ is given by $\|\mu\|=\mu(X)$. Additionally, this allows us to equip $\ba(X)$ with the weak$^{*}$-topology, the coarsest topology making all {\it evaluation maps} $\mu\longmapsto\int\ f\ d\mu$ continuous, for all $f\in\ell^\infty(X)$.

\begin{defn}
    For an infinite set $X$, we will say that a finitely additive measure $\mu$ on $X$ is an {\bf ultrafilter measure} if it is a positive finitely additive measure such that, for all $A\subseteq X$, we have $\mu(A)\in\{0,1\}$.
\end{defn}

We formally record the following easy-to-prove result, while at the same time we introduce some notation that will be used throughout the rest of the paper.

\begin{thm}
    Let $X$ be an infinite set. Given an ultrafilter $u$ on $X$, we define its {\bf ultrafilter measure} $\mu_u$ by:
\[
\mu_{u}(A)=\begin{cases}
			1, & \text{if $A\in u$ }\\
            0, & \text{if A$\notin$ u}.
		 \end{cases}
\]
Then $\mu_u(A)$ is a finitely additive ultrafilter measure. Conversely, given a finitely additive ultrafilter measure $\mu$, we may define $u_\mu=\{A\subseteq X\big|\mu(A)=1\}$, and $u_\mu$ is an ultrafilter on $X$. Furthermore, $\mu_{u_\mu}=\mu$ for all ultrafilter measures, and $u_{\mu_u}=u$ for all ultrafilters $u\in\beta X$.
\end{thm}

Therefore, the family of all ultrafilters and the family of all ultrafilter measures are in a natural bijection. Ultrafilter measures sit within the unit ball of $\ba(X)$ and, furthermore, the aforementioned identification of ultrafilters and ultrafilter measures allows us to think of $\beta X$ as a topological subspace of $\ba(X)$, when we see the latter as equipped with the weak$^*$-topology.

Now, when we perform all of the previous constructions on a set that comes equipped with a group operation, such as if $X=\mathbb Z$ (the additive group of integers), we can identify each $n\in\mathbb Z$ with the basic sequence $e_n\in\ell^1(\mathbb Z)$ (the sequence whose $n$-th value is $1$, and all other values are $0$) and consider the usual convolution operation in $\ell^1(\mathbb Z)$ as a semigroup operation that extends the addition in $\mathbb Z$. This makes $\ell^1(\mathbb Z)$ into a Banach algebra, and induces a Banach algebra structure into its dual space $\ell^\infty(\mathbb Z)$. This viewpoint allows us to define the {\it left Arens product} on $\ba(\mathbb Z)$ by the equation:
\begin{equation*}
    (\mu*\nu)(A)=\int\int\chi_A(n+m)\ d\nu(m)\ d\mu(n);
\end{equation*}
this product makes $\ba(\mathbb Z)$ into a Banach algebra, and coincides with the ultrafilter sum in the case of ultrafilter measures.

%*************************************************************************

We now proceed to state the definitions from additive combinatorics that we will use.

\begin{defn}
    Let $(m_{k})_{k=1}^{\alpha}$, where $\alpha\in\mathbb N\cup\{\infty\}$, be a (finite or infinite) sequence of natural numbers. Define the set of all finite sums of distinct elements of this sequence by
    \[
    \fs(m_{k})=\{m_{k_{1}}+m_{k_{2}}+\dots+m_{k_{m}}| 1\leq k_{1}<\dots<k_{m}, 1\leq m\leq\alpha\}.
    \]
\end{defn}

\begin{defn}
    Let $X\subseteq\mathbb Z$.
    \begin{enumerate}
        \item For an $n\in\mathbb N$, we say that $X$ is an {\bf $\ip_n$-set} if there exists a sequence of length $n$, $(m_k)_{k=1}^n$, such that $\fs(m_k)\subseteq X$,
        \item we say that $X$ is an {\bf $\ip$-set} if there exists an infinite sequence, $(m_k)_{k=1}^\infty$, such that $\fs(m_k)\subseteq X$.
    \end{enumerate}    
\end{defn}

So the Ramsey-theoretic result known as the Folkman--Rado--Sanders's theorem establishes that, for every $n\in\mathbb N$, there exists a sufficiently large $M$ such that whenever $\{1,2,\ldots,M\}=X_0\cup X_1$, then necessarily one of the $X_i$ is an $\ip_n$-set. The more general Hindman's theorem asserts that the family of $\ip$-sets forms a coideal (that is, if $X$ is an $\ip$-set and $X=X_0\cup X_1$, then one of the $X_i$ is an $\ip$-set). The well-known ultrafilter proof of Hindman's theorem proceeds by establishing that there are nonprincipal idempotent ultrafilters (with respect to the ultrafilter addition, that is, there are nonprincipal ultrafilters $u$ such that $u+u=u$), and that every element of a nonprincipal idempotent ultrafilter is an $\ip$-set~\cite[Theorems 5.12 and 16.4]{Hindman-Strauss}. Therefore, translating this to the language of finitely additive measures, we obtain the following.

\begin{cor}\label{prop:idempotentip}
    If $\mu\in\ba(\mathbb Z)$ is an idempotent ultrafilter measure, then every $A\subseteq\mathbb Z$ such that $\mu(A)>0$ is an $\ip$-set.
\end{cor}
%***********************************************************

In the previous corollary, ``idempotent'' refers to the Arens product. So, idempotent finitely additive measures are those measures $\mu\in\ba(\mathbb Z)$ such that $\mu*\mu=\mu$. Combinatorially, recovering the definition of $\mu*\mu$, this means that, for every $A\subseteq\mathbb Z$,
\begin{equation*}
    \mu(A)=\int\mu(A-x)\ d\mu(x).
\end{equation*}

We finish this section with a lemma that will be used throughout the paper.

\begin{prop}\label{characterizeipn}
Let $A\subseteq\mathbb Z$, and let $n\in\mathbb N$. Then, $A$ is an $\ip_{n+1}$-set if and only if there exists an $x\in A$ such that $A\cap(A-x)$ is an $\ip_{n}$-set.
\end{prop}

\begin{proof}\hfill
\begin{description}
    \item[$\Rightarrow$] Let $(x_k)_{k=1}^{n+1}$ be a sequence of length $n+1$ such that $\fs(x_k)\subseteq A$. Then $x_{n+1}\in A$ and, for each $y\in\fs(x_k)_{k=1}^n$, we have that $y+x_{n+1}\in\fs(x_k)_{k=1}^{n+1}\subseteq A$, which means $y\in A-x_{n+1}$. Thus $\fs(x_k)_{k=1}^n\subseteq A-x_{n+1}$; of course it is also the case that $\fs(x_k)_{k=1}^n\subseteq\fs(x_k)_{k=1}^{n+1}\subseteq A$. Therefore $\fs(x_k)_{k=1}^n\subseteq A\cap(A-x_{n+1})$, and so $A\cap(A-x_{n+1})$ is an $\ip_n$-set.
    \item[$\Leftarrow$] Suppose $x\in A$ is such that $A\cap(A-x)$ is an $\ip_n$-set, and let $(x_k)_{k=1}^n$ be a sequence of length $n$ such that $\fs(x_k)\subseteq A\cap(A-x)$. Define $x_{n+1}=x$ and consider the sequence of lenth $n+1$ given by $(x_k)_{k=1}^{n+1}$. Given a $y\in\fs(x_k)_{k=1}^{n+1}$, there are three possible cases:
    \begin{enumerate}
        \item $y\in\fs(x_k)_{k=1}^n$, in which case we have $y\in A\cap(A-x)\subseteq A$;
        \item $y=x_{n+1}=x\in A$,
        \item $y=z+x_{n+1}$, where $z\in\fs(x_k)_{k=1}^n$. In this case we know by hypothesis that $z\in A\cap(A-x)\subseteq A-x$, meaning that $y=z+x_{n+1}=z+x\in A$.
    \end{enumerate}
    In all three cases under consideration we were able to conclude that $y\in A$, therefore $\fs(x_k)_{k=1}^{n+1}\subseteq A$ and $A$ is an $\ip_{n+1}$-set.
\end{description}
\end{proof}
%********************************************************

\section{Idempotent measures and $\ip_n$-sets}
\label{IdempotentandIPset}

There are several different proofs of Hindman's theorem, involving several different ideas: from purely combinatorial~\cite{baumgartner-short-proof-of-hindman}, to using ultrafilters~\cite{Hindman-Strauss,fernandez-monthly}, to using tools such as ergodic theory~\cite{FurstenbergWeiss}. Originally motivated by finding yet another proof utilizing $\ba(\mathbb Z)$, we would like to see what kinds of results, in a spirit similar to that of Corollary~\ref{prop:idempotentip}, can be obtained when considering general idempotent finitely additive measures, that are not necessarily ultrafilter measures. At first sight, we do not seem to be able to obtain results as strong.
%{\color{blue}But we proved this in part 3 in example \ref{example-combination-ultrafilters}!} {\color{red}I think this example is a particular example, but we still have no general theorems (other than~\ref{sequencebounds}) that apply to all idempotent finitely additive measures...}

%*

%{\color{red}The following is from Hindman and Strauss, so it is probably not worth including the proof here. Also, we need to see if we will really use the lemma (otherwise, delete it).
%\begin{lem}
%\label{removingxofA,isIP}
%\cite{Hindman-Strauss}
%Let $A\in \mathbb N$. $A$ is an $IP_{n+1}$-set; if and only if for every $x$ coming from the generating sequence, $A\cap(A-x)$ is an $IP_{n}$-set.
%\end{lem}
%}

%So, in terms of general idempotent measures (that are not necessarily ultrafilter measures), we can at least get the following. 

\begin{thm}\label{sequencebounds}
Recursively define a sequence $(\alpha_n)_{n\in\mathbb N}$ of real numbers by:
\begin{eqnarray*}
    \alpha_1 & = & 0, \\
    \alpha_{n+1} & = & \frac{\alpha_n -1 +\sqrt{\alpha_n^2-2\alpha_n+5}}{2}.
\end{eqnarray*}
Then we have:
\begin{enumerate}
    \item Each of the $\alpha_n$ is a well-defined real number,
    \item the sequence $(\alpha_n)_{n\in\mathbb N}$ is strictly increasing, and for each $n$, $0\leq\alpha_n<1$,
    \item if $\mu$ is an idempotent positive measure in the unit ball of $\ba(\mathbb Z)$, then for every $A\subseteq\mathbb Z$ we have that $\mu(A)>\alpha_n$ implies that $A$ is an $\ip_n$-set.
\end{enumerate}
\end{thm}

\begin{proof}\hfill
\begin{enumerate}
    \item The polynomial $X^2-2X+5$ has a negative discriminant and hence it has no real roots. Therefore, this polynomial always takes positive values. In particular $\alpha_n^2-2\alpha_n+5>0$.
    \item Begin by noting that $(1-\alpha_n)^2<\alpha_n^2-2\alpha_n+5$; from here it follows that $1-\alpha_n<\sqrt{\alpha_n^2-2\alpha_n+5}$ and so the numerator from the inductive definition of $\alpha_{n+1}$ is positive; hence $\alpha_{n+1}>0$. We now prove $\alpha_n<\alpha_{n+1}<1$ by induction on $n$, there being nothing to do for $n=1$. So we assume by induction hypothesis that $\alpha_n<1$, which in particular implies that $4\alpha_n<4$. From here, on the one hand it follows that $(\alpha_n+1)^2<\alpha_n^2-2\alpha_n+5$, hence $\alpha_n+1<\sqrt{\alpha_n^2-2\alpha_n+5}$ and so $2\alpha_n<\alpha_n-1+\sqrt{\alpha_n^2-2\alpha_n+5}$, which immediately implies $\alpha_n<\alpha_{n+1}$. On the other hand, from the same inequality one can deduce $\alpha_n^2-2\alpha_n+5<\alpha_n^2-6\alpha_n+9=(\alpha_n-3)^2$, so that from there it follows that $\sqrt{\alpha_n^2-2\alpha_n+5}<3-\alpha_n$ and hence $\alpha_n-1+\sqrt{\alpha_n^2-2\alpha_n+5}<2$, so that $\alpha_{n+1}<1$.
    \item The proof goes by induction on $n\in\mathbb N$, the case $n=1$ being straightforward ($\mu(A)>0$ implies $A\neq\varnothing$, which is the same as being an $\ip_1$-set). Now suppose the theorem is established for certain $n$, assume that $\mu(A)>\alpha_{n+1}$, and define
    \begin{equation*}
        A^\star=\{x\in A\big|\mu(A-x)>\alpha_n+1-\mu(A)\},
    \end{equation*}
    We will proceed to estimate the measure of $A^\star$. Note that
    \begin{eqnarray*}
        \mu(A) & = & (\mu*\mu)(A)=\int_{\mathbb Z}\mu(A-x)\ d\mu(x) \\
        & = & \int_{\mathbb Z\setminus A}\mu(A-x)\ d\mu(x)+\int_{A\setminus A^\star}\mu(A-x)\ d\mu(x)+\int_{A^\star}\mu(A-x)\ d\mu(x) \\
        & \leq & (1-\mu(A))+(\mu(A)-\mu(A^\star))(\alpha_n+1-\mu(A))+\mu(A^\star) \\
        & = & 1+\alpha_n\mu(A)-\mu(A)^2+\mu(A^\star)(\mu(A)-\alpha_n).
    \end{eqnarray*}
    From here we can deduce that
    \begin{equation*}
        \mu(A^\star)\geq\frac{\mu(A)^2+\mu(A)(1-\alpha_n)-1}{\mu(A)-\alpha_n}
    \end{equation*}
    Note that the denominator of the last expression is positive. On the other hand, by definition, $\alpha_{n+1}$ is the bigger root of the polynomial $X^2+(1-\alpha_n)X-1$, hence the fact that $\mu(A)>\alpha_{n+1}$ implies that the numerator from the last expression is also positive. We may thus conclude that $\mu(A^\star)>0$ and, in particular, $A^\star\neq\varnothing$ so we take an $x\in A^\star$. This means, in particular, that $x\in A$, and moreover
    \begin{eqnarray*}
        \mu(A)+\mu(A-x)-\mu(A\cap(A-x)) & = & \mu(A\cup(A-x)) \\
        & \leq & 1,
    \end{eqnarray*}
    so that (using the fact that $x\in A^\star$)
    \begin{eqnarray*}
        \mu(A\cap(A-x)) & \geq & \mu(A)+\mu(A-x)-1 \\
        & > & \mu(A)+\alpha_n+1-\mu(A)-1 \\
        & = &  \alpha_n.
    \end{eqnarray*}
    Therefore, by induction hypothesis, the set $A\cap(A-x)$ is an $\ip_n$-set. By Lemma~\ref{characterizeipn}, $A$ is an $\ip_{n+1}$-set, and we are done.
\end{enumerate}
\end{proof}
%%**********************************************
%{\color{blue} I think this bound is not optimal since it is coming from an inductive inequality. The recursion for $\alpha_{n}$ comes from an inequality like $\mu(A)>\alpha_{n}$, so $\mu(\{x\colon \mu(A-x)>\alpha_{n-1}\})>0$. The constant is coming from a proof technique, not a combinatorial extermal example.}
%%**********************************************

The sequence from the previous theorem seems interesting in its own right. Note, for example, that $a_1=\frac{-1+\sqrt{5}}{2}$ is the absolute value of the ``negative'' golden ratio (i.e. the negative root of the polynomial $X^2-X-1$).

We do not know if the bounds from Theorem~\ref{sequencebounds} are optimal, in the sense that there couldn't be a theorem if we were to take smaller $\alpha_n$. We do know that for some particular cases of idempotent finitely additive measures (that are not ultrafilter measures), much less is needed, cf. Example~\ref{example-combination-ultrafilters} below.

%%%%%%%%%%%%%%%%%%%%%%%%%%%%%%%%%%%%%%%%%%%%%%
%{\color{blue}This bound $\alpha_{n}$ is not optimal. Define $B=A\cap(A-x)$, so $\mu(B)>\alpha_{n}$. Using idempotence $\mu(A)=\int \mu(A-t)\,d\mu(t)$, then we estimate for every $t$,
%\[
%\mu(A\cap(a-t))\geq\mu(A)+\mu(A-t)-1,
%\]
%so $\mu(B) >2\mu(A)-1>\alpha_{n}$ {\color{red}Aren't you assuming here that $\mu(A-t)=\mu(A)$? But $\mu$ is not translation invariant, merely idempotent.}, why $\mu(B)>\alpha_{n}$;
%we would need $2\mu(A)-1>\alpha_{n}$ I mean $\mu(A)>\frac{\alpha_{n}+1}{2}$, $\Phi(\alpha_{n+1})=\alpha_{n}$, if $\mu(A)>\alpha_{n+1}$, then $\mu(B)\geq\Phi(\mu(A))>\Phi(\alpha_{n+1})=\alpha_{n}$. We know $\alpha_{1}=0,\quad \alpha_{2}=\frac{-1+\sqrt{5}}{2}\approx 0.618$ thus $\alpha_{2}^{\text{true}}<0.618$.\\
%\[
%\mu(A)=\mu\,(\{t |\quad\mu(A-t)\approx \mu(A)\})
%\]
%so for many $t,$
%$\mu(A\cap(A-t))>>2\,\mu(A)-1$ instead of $\mu(B)\geq2\,\mu(A)-1$, we get $\mu(B)\geq\, 2\,\mu(A)-1+\epsilon$ for some $\epsilon>0,$ so $\alpha_{n+1}^{\text{true}}<\alpha_{n+1}$, and this is a contradiction to optimality. Why?\\
%If $\alpha_{n}$ were optimal, then no smaller number could work, but now we have $\alpha_{n}^{\text{true}}$ so $\alpha_{n}$ was not minimal and this directly contradicts the definition of optimality.\\

%Why $\alpha_{n+1}^{\text{true}}<\alpha_{n+1}$?
%I do not know if it is necessary to write!
%}
%%%%%%%%%%%%%%%%%%%%%%%%%%%%%%%%%%%%%%%%%%%%%%

The rest of the section is devoted to examples of idempotent measures that are not ultrafilter measures. We will need to recall some ultrafilter theory for this. The following definitions are standard.

\begin{defn}
Given a semigroup $S$ (equipped with semigroup operation $*$),
\begin{enumerate}
    \item A subset $L\subseteq S$ is called a {\bf left ideal} if $ S*L\subseteq L$,
    \item a left ideal $L\subseteq S$ is {\bf minimal} if, whenever $L'$ is a left ideal such that $L'\subseteq L$, it must be the case that $L=L'$,
    \item given two idempotent elements $u,v\in S$, we say that $u\leq v$ if and only if $u*v=v*u=u$,
    \item an idempotent $u\in S$ is {\bf minimal} if it is minimal with respect to the relation defined above.
\end{enumerate}
\end{defn}

It turns out that the relation $\leq$ is a partial order~\cite[Remark 1.35]{Hindman-Strauss}. Furthermore, for an idempotent $u$, being a minimal idempotent is equivalent to belonging to some minimal left ideal~\cite[Theorem 1.38]{Hindman-Strauss}. Furthermore, if $u$ is a (minimal) idempotent that belongs to the minimal left ideal $L$, then~\cite[Lemma 1.52]{Hindman-Strauss} $L=Lu$ and, furthermore, for each $x\in L$ we have $xu=x$. Another source with a quick but complete introduction to the theory of minimal idempotents and minimal ideals is~\cite{Stevo2010}. The following lemma will be crucial for the construction of our examples.

\begin{lem}\label{lem:absorbing-idempotents}
    There exists a sequence  $(u_n)_{n\in\mathbb N}$ of pairwise distinct nonprincipal idempotent ultrafilters, such that for all $n,m\in I$,
    \begin{equation*}
        u_n+u_m=u_n.
    \end{equation*}
\end{lem}

\begin{proof}
    Working in the (right-topological) semigroup $\mathbb Z^*=\beta\mathbb Z\setminus\mathbb Z$, we pick a minimal left ideal $I$. By~\cite[Theorem 6.9]{Hindman-Strauss}, $I$ contains $2^{\mathfrak c}$ distinct minimal left ideals, so we may arbitrarily choose the $u_n$ to be any pairwise distinct family of minimal left ideals of $I$. By the properties of minimal idempotents within minimal left ideals (explained in the paragraph prior to the present theorem), for each of these $u_m\in L$, and for all $x\in L$, we will have $x+u_m=x$; in particular $u_n+u_m=u_n$ for every $n\in\mathbb N$. 
\end{proof}

%***************************************
\begin{cor}\label{convexcombinations}
    There is an infinite sequence of pairwise distinct idempotent ultrafilter measures $(\mu_n)_{n\in\mathbb N}$ such that, for all $n,m\in\mathbb N$, we have
    \begin{equation*}
        \mu_n*\mu_m=\mu_n.
    \end{equation*}
\end{cor}

\hfill$\Box$

 \begin{ex}\label{example-combination-ultrafilters}
     A (positive) finitely additive measure $\mu$ satisfying:
%     \begin{equation*}
%         \mu_n*\mu_m=\mu_m,
%     \end{equation*}
%     and a sequence of (nonzero) numbers $p_n\in[0,1]$ such that $\sum_{n=1}^{\infty}p_{n}=1$. Define the finitely additive probability measure
%     \[
%     \mu=\sum_{n=1}^{\infty}\,p_{n}\mu_{n},\quad,
%     \]
%Then, 
\begin{enumerate}
    \item $\mu$ is {\bf not} an ultrafilter measure,
    \item $\mu$ is an idempotent measure,
    \item for every $A\subseteq\mathbb Z$, if $\mu(A)>0$ then $A$ is an $\ip$-set.
\end{enumerate} 
\end{ex}

\begin{proof}
It suffices to take $\mu=\frac{1}{2}\mu_1+\frac{1}{2}\mu_2$, where $\mu_1,\mu_2$ are taken from the sequence guaranteed by Corollary~\ref{convexcombinations}.
    \begin{enumerate}
        \item Since $\mu_1,\mu_2$ are two distinct ultrafilter measures, there exists a set $A\subseteq\mathbb Z$ such that
\[
\mu_{1}(A)=1, \quad \mu_2(A)=0,
\]
which implies that $\mu(A)=\frac{1}{2}\notin\{0,1\}$. Therefore $\mu$ is not $\{0,1\}$-valued ---it is not an ultrafilter measure.

\item To see that $\mu$ is an idempotent measure, compute:
% \begin{eqnarray*}
%     (\mu*\mu)&=&\big( \sum_{n}\,p_{n}\,\mu_{n}\big)*\big( \sum_{m}\,p_{m}\,\mu_{m}\big)\\
  %   &=& \sum_{n,m}\,p_{n}\,p_{m}(  \mu_{n}*\mu_{m}) 
%\end{eqnarray*}
%whenever $\mu_{n}*\mu_{m}$ are idempotent, since idempotents absorb each other on the left,
%\[
%\mu_{n}*\mu_{m}=\mu_{n}
%\]
\begin{eqnarray*}
    (\mu*\mu)& = & \left(\frac{1}{2}\mu_1+\frac{1}{2}\mu_2\right)*\left(\frac{1}{2}\mu_1+\frac{1}{2}\mu_2\right)\\
    &=& \frac{1}{4}\mu_1*\mu_1+\frac{1}{4}\mu_1*\mu_2+\frac{1}{4}\mu_2*\mu_1+\frac{1}{4}\mu_2*\mu_2\\
&=& \frac{1}{4}\mu_1+\frac{1}{4}\mu_1+\frac{1}{4}\mu_2+\frac{1}{4}\mu_2\\
& = & \mu,
\end{eqnarray*}
hence $ \mu $ is idempotent.\\

\item Suppose $0<\mu(A)=\frac{1}{2}\mu_{1}+\frac{1}{2}\mu_2$, then there exists at least one $i\in\{1,2\}$ such that $\mu_i(A)>0$. Since $\mu_i$ is an idempotent ultrafilter measure, we conclude that $A$ is an $\ip$-set by Corollary~\ref{prop:idempotentip}.
%\sum_{n=1}^\infty p_n\mu_n(A)$. Then we must have $\mu_i(A)>0$ for some $i$; since the $\mu_i$ are idempotent measures, we conclude that $A$ is an $\ip$-set by Corollary~\ref{idempotent-ultrafilter}.
    \end{enumerate}
\end{proof}

\begin{rem}
    Although the previous example shows an idempotent measure that is a linear combination of idempotent ultrafilters, not every linear (or even convex) combination of idempotent ultrafilters results in an idempotent finitely additive measure. For example, if we were to take two distinct ultrafilters such that $u\leq v$, and we let $\mu_u,\mu_v$ be the corresponding ultrafilter measures, then the measure $\mu=\frac{1}{2}\mu_u+\frac{1}{2}\mu_v$ satisfies:
    \begin{equation*}
        \mu*\mu=\frac{3}{4}\mu_u+\frac{1}{4}\mu_v\neq\mu
    \end{equation*}
    and so $\mu$ is not idempotent.
\end{rem}

In particular, thanks to the previous example we know there exist idempotent finitely additive measures that are not ultrafilter measures. This example is still, however, a linear combination of ultrafilter measures. So the next example will provides us with idempotent finitely additive measures that do not belong to the linear span of ultrafilter measures.

%%%%%%%%%%%%%%%%%%%%%%%%%%%%%%%%%%%%%%%%%%%%%%%%%%%

\begin{ex}\label{infinite-linear-comb}
     A (positive) finitely additive measure $\mu$ satisfying:
\begin{enumerate}
    \item $\mu$ does {\bf not} belong to the linear span of ultrafilter measures,
    \item $\mu$ is an idempotent measure,
    \item for every $A\subseteq\mathbb Z$, if $\mu(A)>0$ then $A$ is an $\ip$-set.
\end{enumerate}
\end{ex}

\begin{proof}
Take a sequence of idempotent ultrafilter measures $(\mu_n)_{n\in\mathbb N}$, satisfying $\mu_n*\mu_m=\mu_n$ for all $n,m$, by Corollary~\ref{convexcombinations}. Now define
\begin{equation*}
\mu=\sum_{n=1}^{\infty}\frac{1}{2^{n}}\mu_{n}   
\end{equation*}
Then $\mu$ is as required. To see this, we first check that it is an idempotent measure:
\begin{eqnarray*}
    \mu*\mu & = & \left(\sum_{n=1}^{\infty}\frac{1}{2^{n}}\mu_{n}\right)*\left(\sum_{m=1}^{\infty}\frac{1}{2^{m}}\mu_{m}\right) \\
    & = & \sum_{n=1}^{\infty}\sum_{m=1}^{\infty}\frac{1}{2^{n+m}}\mu_{n}*\mu_{m} \\
    & = & \sum_{n=1}^{\infty}\sum_{m=1}^{\infty}\frac{1}{2^{n+m}}\mu_{n} \\
    & = & \sum_{n=1}^{\infty}\left(\sum_{m=1}^{\infty}\frac{1}{2^{n+m}}\right)\mu_{n} \\
    & = & \sum_{n=1}^{\infty}\frac{1}{2^{n}}\mu_{n} \\
    & = & \mu.
\end{eqnarray*}
Furthermore, the third property is easy to check, for if $\mu(A)>0$ then there is an $n$ such that $\mu_n(A)>0$; since $\mu_n$ is an idempotent ultrafilter measure, $A$ must be an $\ip$-set. Finally, it remains to prove that $\mu$ cannot be written as a (finite) linear combination of ultrafilter measures, so suppose, to the contrary, that $\mu=\sum_{i=1}^k a_i \nu_i$, where each $\nu_i$ is an ultrafilter measure. From this assumption, a simple calculation shows that, for all $A\subseteq\mathbb Z$, if $\mu(A)>0$ then $\mu(A)>\min\{a_i\big|1\leq i\leq k\}$. Let $N$ be large enough that $\frac{1}{2^{N-1}}<\min\{a_i\big|1\leq i\leq k\}$, suppose that $u_i$ is the ultrafilter corresponding to the measure $\mu_i$ for $i\leq N$, and pick a set $A$ such that $A\in u_N$ but $A\notin u_i$ for $i<N$ (this can be done by means of a relatively straightforward induction, or one can use e.g.~\cite[Lemma 3.4]{fernandez-navarro-soria}, which guarantees even more). This means that $\mu_i(A)=0$ for $i<N$, but $\mu_N(A)=1$. Hence, 
\begin{equation*}
    \mu(A)=\sum_{n=1}^\infty \frac{1}{2^n}\mu_n(A)=\frac{1}{2^N}+\sum_{n=N+1}^\infty\frac{1}{2^n}\mu_n(A),
\end{equation*}
so that $0<\frac{1}{2^N}\leq\mu(A)\leq\frac{1}{2^{N-1}}$, a contradiction.
\end{proof}

\section{Translation-invariant and almost translation-invariant measures}
\label{Translationinvariantmeasures}

In this section we introduce two kinds of measures that provide stronger combinatorial properties for subsets of $\mathbb Z$.

\begin{defn}
    Let $\mu\in\ba(\mathbb Z)$.
    \begin{enumerate}
        \item We say that $\mu$ is {\bf translation-invariant} if, for every $A\subseteq\mathbb Z$ and for every $x\in\mathbb Z$, $\mu(A-x)=\mu(A)$.
        \item We say that $\mu$ is {\bf almost translation-invariant} if, for every $A\subseteq\mathbb Z$, we have
    \begin{equation*}
        \mu\left(\left\{x\in\mathbb Z\big|\mu(A-x)=\mu(A)\right\}\right)=1.
    \end{equation*}
    \end{enumerate}
\end{defn}

Translation-invariant finitely additive probability measures,also known as invariant means, have been extensively studied on topological groups and are closely related to amenability~\cite{F.Greenleaf}. In the particular setting of the natural numbers, translation-invariant finitely additive probability measures and their basic properties have been studied~\cite{Gardner}. It is not extremely hard to obtain these kinds of measures by means of the Hahn--Banach theorem. The notion of almost translation-invariant measures seems to be a natural generalization, that happens to be related to idempotence: an ultrafilter measure is almost translation-invariant if and only if it is idempotent~\cite{hindman-ultrafiltersch}.
%***********************************

Clearly, every translation-invariant measure is almost translation-invariant. It is worth noting that translation-invariant measures are never ultrafilter measures, for if we partition $\mathbb Z=A_0\cup A_1$, say, letting $A_0$ be the set of even numbers and $A_1$ the set of odd numbers, then if $\mu$ is translation invariant we have $\mu(A_1)=\mu(1+A_0)=\mu(A_0)$, so $2\mu(A_0)=\mu(A_0)+\mu(A_1)=\mu(A_0\cup A_1)=\mu(\mathbb Z)$. Therefore, if $\mu$ is nonzero, then necessarily $\mu(A_0)\neq 0$ and $\mu(\mathbb Z)=2\mu(A_0)$, showing that $\mu$ cannot be $\{0,1\}$-valued (this argument was probably first explicitly pointed out by Hindman~\cite{hindman-ultrafiltersch}).

In particular, idempotent measures are not necessarily translation invariant (since any idempotent ultrafilter measure cannot be translation-invariant). Conversely, a translation-invariant finitely additive measure that is not an ultrafilter measure is not idempotent. Since, if $\mu$ is such a measure, and $A\subseteq\mathbb Z$ satisfies $0<\mu(A)<1$, then, if $\mu$ were idempotent we would have:
    \begin{eqnarray*}
        \mu(A) & = & \int \mu(A-x)\,d\mu(x)\\
        &=&\int \mu(A)\,d\mu(x)\\
        &=& \mu(A)\int d\mu(x)\\
        &=& \mu(A)^{2},
    \end{eqnarray*}
    which can only happen if $\mu(A)\in\{0,1\}$, a contradiction.

Now, it is well-known that if $\mu$ is an ultrafilter measure then $\mu$ is almost translation-invariant if and only if it is idempotent. So almost translation-invariant measures appear to be a natural generalization of both translation-invariant and idempotent measures. In the rest of the section, we explore the combinatoral implications of almost translation-invariant measures first, and of translation-invariant measures after. We begin by noting that almost translation-invariant measures yield much better lower bounds than the ones obtained from idempotent measures in Theorem~\ref{sequencebounds}.

\begin{thm}\label{almosttranslationinvariant}
Let $\mu$ be a positive, finitely additive measure, belonging to the unit ball of $\ba(\mathbb Z)$, that is almost translation-invariant. Then, for every $A\subseteq\mathbb Z$, if $\mu(A)>1-\frac{1}{2^{n-1}}$ then $A$ is an $\ip_{n}$-set.
\end{thm}
\begin{proof}
We proceed by induction on $n$, the case $n=1$ being trivial. Suppose the result is true for $n$, and let $A\subseteq\mathbb Z$ be such that $\mu(A)>1-\frac{1}{2^n}$. Letting $A^*=\{x\in\mathbb Z\big|\mu(A-x)=\mu(A)\}$, we know by almost translation-invariance that $\mu(A^*)=1$. In particular, $A\cap A^*\neq\varnothing$, so let $x\in A\cap A^*$. Since $\mu(A-x)=\mu(A)$, we obtain
\begin{eqnarray*}
    1\geq \mu(A\cup(A-x)) & = & \mu(A)+\mu(A-x)-\mu(A\cap(A-x)) \\
    & = & 2\mu(A)-\mu(A\cap(A-x)) \\
    & > & 2-\frac{1}{2^{n-1}}-\mu(A\cap(A-x)).
\end{eqnarray*}
We may thus deduce that $\mu(A\cap(A-x))>1-\frac{1}{2^{n-1}}$; by induction hypothesis, $A\cap(A-x)$ is an $\ip_n$-set, and so $A$ is an $\ip_{n+1}$-set by Lemma~\ref{characterizeipn}.
\end{proof}

We now introduce a stronger combinatorial definition that will be related to translation-invariant measures; this definition is inspired in
%\section{Translation-invariant measures}
%\label{Translationinvariantmeasures}
%Inspired by
the characterization of IP-sets given by Proposition~\ref{characterizeipn}.%, we introduce the following definition.

\begin{defn}
    Given an $n\in\mathbb N$, we say that a set $A\subseteq\mathbb Z$ is {\bf strongly $\ip_n$}, abbreviated $\sip_n$, in the following recursive way:
    \begin{enumerate}
        \item $A$ is $\sip_1$ if it is nonempty,
        \item $A$ is $\sip_{n+1}$ if, for every $a\in A$, the set $A\cap(A-n)$ is $\sip_n$.
    \end{enumerate}
\end{defn}

\begin{rem}
    Proposition~\ref{characterizeipn} shows that every $\sip_n$ set must also be an $\ip_n$ set. However, for $n\geq 2$, the two notions are no longer equivalent. For example, the (finite) set $\{1,2,3\}$ is an $\ip_2$ set (it contains $\fs(1,2)$), but it is not a $\sip_2$ set. This is because every $\sip_2$ set must be infinite: for, if $A$ is $\sip_2$, then one can recursively build an infinite, injective sequence $(a_n)_{n\in\mathbb N}$ of elements of $A$, simply by letting $a_0\in A$ be arbitrary and, knowing $a_n\in A$, there must be $x_n\in A\cap(A-a_n)$ (since the latter set must be $\sip_1$ and, hence, nonempty) so we may define $a_n<a_n+x_n=a_{n+1}\in A$.
\end{rem}

The previous remark shows, indeed, that it is possible to have, for every $n$, $\ip_n$ sets that are not even $\sip_2$ (since there are $\ip_n$ sets that are finite). Despite this, the following example shows that there is no implication relation between $\sip_2$ and $\ip_n$ when $n\geq 3$.

\begin{ex}\label{ex:sipnotip}
    There exists an $\sip_2$ set that is not an $\ip_3$ set.
\end{ex}

\begin{proof}
    Choose an increasing sequence of natural numbers $(x_n)_{n\in\mathbb N}$ such that, for all $n\in\mathbb N$, $4x_n<x_{n+1}$. Under this condition, a standard inductive argument shows that $x_{n+1}>2\sum_{i=1}^n x_i$, and this implies that the sequence $(x_n)_{n\in\mathbb N}$ has {\it $2$-uniqueness of finite sums}, as defined in~\cite[Def. 2.5]{strongly-union-trivialsums}; in other words, every element of the form $\varepsilon_1 x_{i_1}+\cdots+\varepsilon_k x_{i_k}$ with $i_1<\cdots<i_k$ and each $\varepsilon_j\in\{1,2\}$ can be represented uniquely as such (in particular, every $a\in\fs(x_n)_{n\in\mathbb N}$ can be represented in a unique (up to the order of the summands) way as a finite sum of elements of the sequence $(x_n)_{n\in\mathbb N}$). Now define
    \begin{equation*}
        A=\left\{a_n+a_{n+1}+\cdots+a_{n+k}\big|n\in\mathbb N\text{ and }k\geq 0\right\},
    \end{equation*}
    in other words, $A$ is the set of {\it adjacent finite sums} from the sequence $(x_n)_{n\in\mathbb N}$, as defined in~\cite{carlucci-adjacent}. %Similarly we define the sets $A_m=\left\{a_n+a_{n+1}+\cdots+a_{n+k}\big|m\leq n\text{ and }k\geq 0\right\}$; in other words, $A_m$ is the set of adjacent finite sums of the truncated sequence $(x_n)_{n=m}^\infty$.
    Then $A$ is an $\sip_2$ set: given any $a\in A$, we write $a=x_n+x_{n+1}+\cdots+x_{n+k}$ and note that
    \begin{equation*}
        \left\{a_{n+k+1}+\cdots+a_{n+k+m}\big|m\in\mathbb N\right\}\subseteq A\cap(A-a),
    \end{equation*}
    and in particular the latter set is nonempty, hence $\sip_1$.

    On the other hand we claim that $A$ is not an $\ip_3$-set. To see this, take any $a=x_n+x_{n+1}+\cdots+x_{n+k}\in A$, and let us analyze what the elements of $A\cap(A-a)$ look like: if $b\in A\cap(A-a)$, then since $b\in a$ we can write $b=x_m+x_{m+1}+\cdots+x_{m+t}$ for some $m,t$; also, since $b\in A-a$ we must have $a+b\in A$ and so $a+b=x_\ell+x_{\ell+1}+\cdots+x_{\ell+s}$ for some $\ell,s$. Therefore
    \begin{equation*}        x_\ell+x_{\ell+1}+\cdots+x_{\ell+s}=a+b=x_n+x_{n+1}+\cdots+x_{n+k}+x_m+x_{m+1}+\cdots+x_{m+t};
    \end{equation*}
    from here, the $2$-uniqueness of finite sums for the sequence $(x_n)_{n\in\mathbb N}$ implies that either $n=\ell$, $m=n+x+1$, and $m+t=\ell+s$, or $m=\ell$, $n=m+t+1$ and $n+x=\ell+s$. Since this holds for every $b\in A\cap(A-a)$, we see that
    \begin{eqnarray*}
        A\cap(A-a) & = & \left\{x_{n-s}+x_{n-s+1}+\cdots+x_{n-1}\big|s<n\right\}\cup \\
        & & \cup\left\{x_{n+k+1}+x_{n+k+2}+\cdots+x_{n+k+s}\big|s\in\mathbb N\right\}.
    \end{eqnarray*}
    Given this description of $A\cap(A-a)$, it becomes easy to prove, using again the $2$-uniqueness of finite sums, that whenever $x,y\in A\cap(A-a)$ we must have $x+y\notin A\cap(A-a)$; hence the latter set is not an $\ip_2$-set. Since $a\in A$ was arbitrary, we conclude that $A$ is not an $\ip_3$-set.
\end{proof}

%***************************************************************
\begin{thm}\label{translationinvariant}
Let $\mu$ be a translation-invariant, positive finitely additive measure belonging to the unit ball of $\ba(\mathbb Z)$. Then, for every $A\subseteq\mathbb Z$, if $\mu(A)>1-\frac{1}{2^{n-1}}$ then $A$ is an $\sip_{n}$-set.
\end{thm}
\begin{proof}
We proceed by induction on $n$. The case $n=1$ is trivial since it obviously $\mu(A)>0$ implies $A\neq\varnothing$, that is, $A$ is $\sip_1$. Now suppose the result is true for $n$, and let $A\subseteq\mathbb Z$ be such that $\mu(A)>1-\frac{1}{2^n}$. Pick any $x\in A$ and consider the set $A-x$. Translation-invariance of $\mu$ implies that $\mu(A-x)=\mu(A)$, from where we can obtain
\begin{eqnarray*}
    1\geq \mu(A\cup(A-x)) & = & \mu(A)+\mu(A-x)-\mu(A\cap(A-x)) \\
    & = & 2\mu(A)-\mu(A\cap(A-x)) \\
    & > & 2-\frac{1}{2^{n-1}}-\mu(A\cap(A-x)).
\end{eqnarray*}
We may thus deduce that $\mu(A\cap(A-x))>1-\frac{1}{2^{n-1}}$; by induction hypothesis, this implies that $A\cap(A-x)$ is an $\sip_n$-set, and so $A$ must be an $\sip_{n+1}$-set by Lemma~\ref{characterizeipn}.
\end{proof}

\section{Discussion}

It is natural to wonder whether the bounds from Theorems~\ref{sequencebounds},~\ref{almosttranslationinvariant}, and~\ref{translationinvariant} are optimal. The examples obtained in this paper do not show this, as all of the measures $\mu$ described here satisfy that, if $\mu(A)>0$, then $A$ is an $\ip$-set. Of course, answering these optimality questions would involve explicitly constructing some carefully fine-tuned finitely additive measures.

\begin{question}
    \begin{enumerate}
        \item Is there a finitely additive positive idempotent measure $\mu\in\ba(\mathbb Z)$, and a set $A\subseteq\mathbb Z$ that is not an $\ip_n$-set, such that $\mu(A)=\alpha_n$? (Where $\alpha_n$ is as defined in Theorem~\ref{sequencebounds}.)
        \item Is there an almost translation-invariant measure $\mu\in\ba(\mathbb Z)$, and a set $A\subseteq\mathbb Z$ that is not an $\ip_n$-set, such that $\mu(A)=1-\frac{1}{2^{n-1}}$?
        \item Is there a translation-invariant measure $\mu\in\ba(\mathbb Z)$, and a set $A\subseteq\mathbb Z$ that is not a $\sip_n$-set, such that $\mu(A)=1-\frac{1}{2^{n-1}}$?
    \end{enumerate}
\end{question}

Finally, we have studied three kinds of finitely additive measures: idempotent, translation-invariant and almost translation-invariant. We know that the first two kinds bear no relationship between them (in the sense that being of one kind does not imply being of the other kind). There is, however, not much known about almost translation-invariant measures, other than the fact that these are equivalent to idempotent in the ultrafilter setting, but not necessarily in general.

\begin{question}
    \begin{enumerate}
        \item Is there an almost translation-invariant measure that is neither an ultrafilter measure, nor translation-invariant?
        \item Is there an almost translation-invariant measure that is not idempotent?
        \item Is there an idempotent finitely additive measure that is not almost translation-invariant?
    \end{enumerate}
\end{question}

In the previous three questions, the first requires the measure not to be an ultrafilter measure in order for the question not to be trivial (since every idempotent ultrafilter measure is almost translation-invariant but not translation-invariant). The other two questions, on the other hand, do not need this explicitly stated, but any examples of a measure answering positively any of the two questions cannot possibly be an ultrafilter measure.

\section*{Acknowledgements}

The second author is grateful to Juris Stepr\={a}ns for introducing him (over 10 years ago!) to this research topic. The first author was supported by a Secihti Postdoctoral Fellowship ({\it Estancias Posdoctorales por México}), under the mentorship of the second author. The second author was partially supported by IPN's internal grant SIP-20260817 as well as by Secihti's grant CBF2023-2024-334.

%%%%%%%%%%%%%%%%%%%%   End of main body of article
%
%                             References
%
%   BiBTeX users uncomment the following line:
%
%\bibliographystyle{jloganal}
%

\end{document}